\documentclass[graybox]{svmult}

\usepackage{type1cm}        % activate if the above 3 fonts are
\usepackage[bottom]{footmisc}% places footnotes at page bottom

\usepackage{amsmath,amsfonts,amssymb}
\usepackage{enumitem}
\usepackage{tikz}
\usepackage{array}
\usepackage{makecell}
\newcolumntype{x}[1]{>{\centering\arraybackslash}p{#1}}
\usepackage{tikz}
\newcommand\diag[4]{%
  \multicolumn{1}{p{#2}|}{\hskip-\tabcolsep
  $\vcenter{\begin{tikzpicture}[baseline=0,anchor=south west,inner sep=#1]
  \path[use as bounding box] (0,0) rectangle (#2+2\tabcolsep,\baselineskip);
  \node[minimum width={#2+3\tabcolsep},minimum height=\baselineskip+\extrarowheight] (box) {};
  \draw (box.north west) -- (box.south east);
  \node[anchor=south west] at (box.south west) {#3};
  \node[anchor=north east] at (box.north east) {#4};
 \end{tikzpicture}}$\hskip-\tabcolsep}}

\makeindex             % used for the subject index
\begin{document}
	
	\title*{On the effect of boundary conditions on the scalability of Schwarz methods}
	
	\author{Gabriele Ciaramella and Luca Mechelli}
	\authorrunning{G. Ciaramella and L. Mechelli}
	\institute{Gabriele Ciaramella \at Politecnico di Milano \email{gabriele.ciaramella@polimi.it}
		\and Luca Mechelli \at Universit\"at Konstanz \email{luca.mechelli@uni-konstanz.de}}
	\maketitle
	
	\abstract*{..}
	\abstract*{In contrast with classical Schwarz theory, recent results have shown that for special domain geometries, 
	one-level Schwarz methods can be scalable. This property has been proved for the Laplace equation and external 
	Dirichlet boundary conditions. Much less is known if mixed boundary conditions are considered.
	This short manuscript focuses on the convergence and scalability analysis of one-level parallel Schwarz method 
	and optimized Schwarz method for several different external configurations of boundary conditions, i.e., 
	mixed Dirichlet, Neumann and Robin conditions.}

%%%%%%%%%%%%%%%%%%%%%%%%%%%%%%%%%%%%%%%%%%%%%%%%%%%%%%%%%%%%%%%%%%%%%%%%%%%%%%%%%%%%%%%%%%%%%%		
\section{Introduction}
\vspace*{-4mm}
	
This work is concerned with convergence and weak scalability\footnote{Here, weak scalability is understood 
in the sense that the contraction  factor does not deteriorate as the number $N$ of subdomains increases and, 
hence, the number of iterations, needed to reach a given tolerance, is uniformly bounded in $N$; see, e.g., \cite{Ciaramella_mini_10_CiaramellaGander4}.} analysis of one-level 
parallel Schwarz method (PSM) and optimized Schwarz method (OSM) for the solution of the problem
\begin{equation}\label{Ciaramella_mini_10_eq:P_2D}
\begin{split}
&-\Delta u = f \, \text{ in $\Omega$}, \quad
u(a_1,y)=u(b_N,y) =0 \: \text{ $y\in(0,1)$}, \\
&\mathcal{B}_b(u)(x)=\mathcal{B}_t(u)(x)=0  \: \text{ $x\in(a_1,b_N),$}
\end{split}
\end{equation}
where $\Omega$ is the domain depicted in Fig.~\ref{Ciaramella_mini_10_fig:1},
and $\mathcal{B}_b$ and $\mathcal{B}_t$ are either Dirichlet, or Neumann or Robin operators:
\begin{align*}
\text{Dirichlet:} \; &\mathcal{B}_{b}(u)(x) = u(0,x), &&\mathcal{B}_{t}(u)(x) = u(1,x), \\
\text{Neumann:} \; &\mathcal{B}_{b}(u)(x) = \partial_y u(0,x), &&\mathcal{B}_{t}(u)(x) = \partial_y u(1,x), \\
\text{Robin:} \, &\mathcal{B}_{b}(u)(x) = q u(0,x) - \partial_y u(0,x), &&\mathcal{B}_{t}(u)(x) = q u(1,x) + \partial_y u(1,x).
\end{align*}
Here, $q>0$ and the subscripts `$b$' and `$t$' stand for `bottom' and `top'.
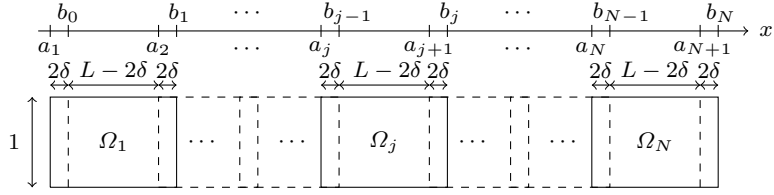
\begin{figure}
\centering
\begin{tikzpicture}[scale=0.8]
% DD grid;
\draw[black,->] (-4.0,-0.6) -- (7.8,-0.6);
\node at (8.10,-0.6) {$x$};

\draw[black] (-3.5,-0.5) -- (-3.5,-0.7);
\node at (-3.5,-0.3) {\small{$b_0$}};
\draw[black] (-3.8,-0.5) -- (-3.8,-0.7);
\node at (-3.8,-0.9) {\small{$a_1$}};
\draw[black] (-2.0,-0.5) -- (-2.0,-0.7);
\node at (-2.0,-0.9) {\small{$a_2$}};
\draw[black] (-1.7,-0.5) -- (-1.7,-0.7);
\node at (-1.65,-0.3) {\small{$b_1$}};

\node at (-0.50,-0.3) {\small{$\cdots$}};
\node at (-0.50,-0.9) {\small{$\cdots$}};

\draw[black] (0.7,-0.5) -- (0.7,-0.7);
\node at (0.7,-0.9) {\small{$a_j$}};
\draw[black] (1.0,-0.5) -- (1.0,-0.7);
\node at (1.15,-0.3) {\small{$b_{j-1}$}};
\draw[black] (2.5,-0.5) -- (2.5,-0.7);
\node at (2.5,-0.9) {\small{$a_{j+1}$}};
\draw[black] (2.8,-0.5) -- (2.8,-0.7);
\node at (2.85,-0.3) {\small{$b_{j}$}};

\node at (4.00,-0.3) {\small{$\cdots$}};
\node at (4.00,-0.9) {\small{$\cdots$}};

\draw[black] (5.2,-0.5) -- (5.2,-0.7);
\node at (5.2,-0.9) {\small{$a_{N}$}};
\draw[black] (5.5,-0.5) -- (5.5,-0.7);
\node at (5.70,-0.3) {\small{$b_{N-1}$}};
\draw[black] (7.3,-0.5) -- (7.3,-0.7);
\node at (7.35,-0.3) {\small{$b_{N}$}};
\draw[black] (7.0,-0.5) -- (7.0,-0.7);
\node at (7.05,-0.9) {\small{$a_{N+1}$}};

\begin{scope}[shift={(0,-3.2)}]
% Subdomains with overlap
\draw[black,<->] (-4.1,0) -- (-4.1,1.5);
\draw[black,dashed] (-3.5,0) -- (-3.5,1.5);
\node at (-4.4,0.75) {$1$};
\draw[black] (-3.8,0) -- (-1.7,0) -- (-1.7,1.5) -- (-3.8,1.5) -- (-3.8,0);
\node at (-2.75,0.75) {$\Omega_1$};
\draw[black,<->] (-3.8,1.7) -- (-3.5,1.7);
\node at (-3.65,1.95) {\footnotesize{$2\delta$}};
\draw[black,<->] (-2.0,1.7) -- (-1.7,1.7);
\node at (-1.85,1.95) {\footnotesize{$2\delta$}};
\draw[black,<->] (-3.5,1.7) -- (-2.0,1.7);
\node at (-2.75,1.95) {\footnotesize{$L-2\delta$}};
%\node at (-2.60,-0.35) {$g_1$};

\draw[black,dashed] (-2.0,0) -- (-0.35,0) -- (-0.35,1.5) -- (-2.0,1.5) -- (-2.0,0);
\node at (-1.25,0.75) {$\cdots$};
\draw[black,dashed] (-0.65,0) -- (1.0,0) -- (1.0,1.5) -- (-0.65,1.5) -- (-0.65,0);
\node at (0.25,0.75) {$\cdots$};

\draw[black] (0.7,0) -- (2.8,0) -- (2.8,1.5) -- (0.7,1.5) -- (0.7,0);
\node at (1.75,0.75) {$\Omega_j$};
\draw[black,<->] (0.7,1.7) -- (1.0,1.7);
\node at (0.85,1.95) {\footnotesize{$2\delta$}};
\draw[black,<->] (2.5,1.7) -- (2.8,1.7);
\node at (2.65,1.95) {\footnotesize{$2\delta$}};
\draw[black,<->] (1.0,1.7) -- (2.5,1.7);
\node at (1.8,1.95) {\footnotesize{$L-2\delta$}};
%\node at (1.75,-0.35) {$g_j$};

\draw[black,dashed] (2.5,0) -- (4.15,0) -- (4.15,1.5) -- (2.5,1.5) -- (2.5,0);
\node at (3.25,0.75) {$\cdots$};
\draw[black,dashed] (3.85,0) -- (5.5,0) -- (5.5,1.5) -- (3.85,1.5) -- (3.85,0);
\node at (4.75,0.75) {$\cdots$};

\draw[black] (5.2,0) -- (7.3,0) -- (7.3,1.5) -- (5.2,1.5) -- (5.2,0);
\node at (6.25,0.75) {$\Omega_N$};
\draw[black,dashed] (7.0,0) -- (7.0,1.5);
\draw[black,<->] (5.2,1.7) -- (5.5,1.7);
\node at (7.15,1.95) {\footnotesize{$2\delta$}};
\draw[black,<->] (7.0,1.7) -- (7.3,1.7);
\node at (5.35,1.95) {\footnotesize{$2\delta$}};
\draw[black,<->] (5.5,1.7) -- (7.0,1.7);
\node at (6.25,1.95) {\footnotesize{$L-2\delta$}};
%\node at (6.10,-0.35) {$g_N$};
\end{scope}
\end{tikzpicture}
\caption{Two-dimensional chain of $N$ rectangular fixed-sized
  subdomains.}
\label{Ciaramella_mini_10_fig:1}
\end{figure}
As shown in Fig.~\ref{Ciaramella_mini_10_fig:1}, the domain $\Omega$ is 
the union of subdomains $\Omega_j$, $j=1,\dots,N$, defined as
$\Omega_j:=(a_j,b_j)\times(0,1)$, where $a_1=0$, $a_j=L+a_{j-1}$
for $j = 2,\dots,N+1$ and $b_j=a_{j+1}+2\delta$ for $j = 0,\dots,N$.
Hence, the length of each subdomain is $L+2\delta$ and the
length of the overlap is $2\delta$ with $\delta \in (0,L/2)$.

It is well known that one-level Schwarz methods are
not weakly scalable, if the number of subdomains increases and the whole domain $\Omega$ is fixed. 
However, the recent work \cite{Ciaramella_mini_10_Stamm3}, published in the field of molecular dynamics,
has drawn attention to the opposite case in which the number of subdomains increases, but their size remains 
unchanged, and, as a result, the size of the whole domain $\Omega$ increases.
In this setting, weak scalability of PSM and OSM for \eqref{Ciaramella_mini_10_eq:P_2D} with
Dirichlet boundary conditions is studied in \cite{Ciaramella_mini_10_CiaramellaGander,Ciaramella_mini_10_CiaramellaGander4}.
Scalability results for the PSM in case of more general geometries of the (sub)domains are presented in
\cite{Ciaramella_mini_10_CiaramellaGander2,Ciaramella_mini_10_CiaramellaGander3,Ciaramella_mini_10_CHS2}.
In these works, only external Dirichlet conditions are discussed and, in such a case, weak scalability is shown. 
A short remark about the non-scalability in case of external Neumann conditions is 
given in \cite{Ciaramella_mini_10_CiaramellaGander4}. 
Similar results have been recently presented in \cite{Ciaramella_mini_10_bootland2020}
for time-harmonic problems.
The goal of this work is to study the effect of different (possibly mixed) external boundary conditions
on convergence and scalability of PSM and OSM. In particular, we will show that only in
the case of (both) external Neumann conditions at the top and the bottom of $\Omega$, PSM and OSM are not scalable.
External Dirichlet conditions lead to the fastest convergence, while external Robin conditions
lead to a convergence that depends heavily on the parameter $q$.

One-level PSM and OSM for the solution of \eqref{Ciaramella_mini_10_eq:P_2D} are
\begin{equation}\label{Ciaramella_mini_10_eq:chainP1_2D}
\begin{split}
&- \Delta u_j^n = f_j \, \text{ in $\Omega_j$}, \\
&\mathcal{B}_b(u_j^n)(x)=\mathcal{B}_t(u_j^n)(x)=0  \: \text{ $x\in(a_1,b_N),$}\\
&\mathcal{T}_{\ell}(u_j^n)(a_j)= \mathcal{T}_{\ell}(u_{j-1}^{n-1})(a_j), 
\quad
\mathcal{T}_r(u_j^n)(b_j) = \mathcal{T}_r(u_{j+1}^{n-1})(b_j) ,
\end{split}
\end{equation}
for $j=2,\dots,N$, where $\mathcal{T}_\ell$ and $\mathcal{T}_r$
are Dirichlet trace operators,
\begin{equation}\label{Ciaramella_mini_10_eq:T_D}
\mathcal{T}_{\ell}(u_j^n)(a_j)=u(a_j,y) \text{ and } \mathcal{T}_r(u_j^n)(b_j)=u(b_j,y),
\end{equation}
for the PSM, and Robin trace operators,
\begin{equation}\label{Ciaramella_mini_10_eq:T_R}
\medmuskip=0.2mu
\thinmuskip=0.2mu
\thickmuskip=0.2mu
\nulldelimiterspace=1.2pt
\scriptspace=1.2pt    
\arraycolsep1.2em
\mathcal{T}_{\ell}(u_j^n)(a_j)=p u(a_j,y) - \partial_x u(a_j,y) \text{ and } \mathcal{T}_r(u_j^n)(b_j)=pu(b_j,y)+ \partial_x u(b_j,y) ,
\end{equation}
with $p>0$ for the OSM. The subscript `$\ell$' and `$r$' stand for `left' and `right'.
For $j=1$ the condition at $a_1$ must be replaced by
$u_1^n(a_1,y) = 0$ and for $j=N$ the condition at $b_N$ must be replaced by
$u_N^n(b_N,y) = 0$.
In this paper, `external conditions' and `transmission conditions' will always refer to the
conditions obtained by the pairs $(\mathcal{B}_b,\mathcal{B}_u)$ and
$(\mathcal{T}_\ell,\mathcal{T}_r)$, respectively.
Note that the Robin parameter $p$ of the OSM can be chosen independently of the
Robin parameter $q$ used for the operators $\mathcal{B}_b$ and $\mathcal{B}_t$.
We analyze convergence of PSM and OSM by a Fourier analysis in Section \ref{Ciaramella_mini_10_sec:BC_scal}. 
For this purpose, we use the solutions of eigenproblems of the 1D Laplace operators with mixed boundary conditions.
These are studied in Section~\ref{Ciaramella_mini_10_sec:diag}.
Finally, results of numerical experiments are presented in Section~\ref{Ciaramella_mini_10_sec:numerics}.

\section{Laplace eigenpairs for mixed external conditions}\label{Ciaramella_mini_10_sec:diag}
\vspace*{-4mm}

Consider the 1D eigenvalue problem
\begin{equation}\label{Ciaramella_mini_10_eq:eig_1D}
\varphi''(y) = -\lambda \varphi(y), \text{ for $y \in (0,1)$, $\quad\mathcal{B}_{b}(\varphi)(0)=\mathcal{B}_{t}(\varphi)(1)=0$},
\end{equation}
and six pairs of boundary operators $(\mathcal{B}_{b},\mathcal{B}_{t})$:
\begin{itemize}[itemsep=0.1em,leftmargin=11mm,labelsep=2.0mm]
\item[{\rm (DD)}] $\mathcal{B}_{b}(\varphi)(0)=\varphi(0)$,
$\quad \qquad \qquad \mathcal{B}_{t}(\varphi)(1)=\varphi(1)$,

\item[{\rm (DR)}] $\mathcal{B}_{b}(\varphi)(0)=\varphi(0)$,
$\quad \qquad \qquad \mathcal{B}_{t}(\varphi)(1)=q\varphi(1)+\varphi'(1)$,

\item[{\rm (DN)}] $\mathcal{B}_{b}(\varphi)(0)=\varphi(0)$,
$\quad \qquad \qquad \mathcal{B}_{t}(\varphi)(1)=\varphi'(1)$,

\item[{\rm (RR)}] $\mathcal{B}_{b}(\varphi)(0)=q\varphi(0)-\varphi'(0)$,
$\quad \mathcal{B}_{t}(\varphi)(1)=q\varphi(1)+\varphi'(1)$,

\item[{\rm (NR)}] $\mathcal{B}_{b}(\varphi)(0)=\varphi'(0)$,
	$\quad \qquad \quad \; \; \, \mathcal{B}_{t}(\varphi)(1)=q\varphi(1)+\varphi'(1)$,
%\item[{\rm (RN)}] $\mathcal{B}_{b}(\varphi)(0)=q\varphi(0)-\varphi'(0)$,
%$\mathcal{B}_{t}(\varphi)(1)=\varphi'(1)$,

\item[{\rm (NN)}] $\mathcal{B}_{b}(\varphi)(0)=\varphi'(0)$,
$\quad \qquad \quad \; \; \, \mathcal{B}_{t}(\varphi)(1)=\varphi'(1)$,
\end{itemize}
where $q>0$ and `D', `R' and `N' stand for `Dirichlet', `Robin' and `Neumann'. 
For all these six cases the eigenvalue problem \eqref{Ciaramella_mini_10_eq:eig_1D}
is solved by orthonormal (in $L^2(0,1)$) Fourier basis functions.

\begin{theorem}[Eigenpairs of the Laplace operator]\label{Ciaramella_mini_10_thm:1}
Let $q>0$. The eigenproblems \eqref{Ciaramella_mini_10_eq:eig_1D} with the above external 
conditions are solved by the non-trivial eigenpairs $(\varphi_k,\lambda_k)$ given by
\begin{itemize}[itemsep=0.1em,leftmargin=11mm,labelsep=2.0mm]
\item[{\rm (DD)}] $\varphi_k(y)=\sqrt{2}\sin(\pi k y)$, $\lambda_k=\pi^2 k^2$, $k=1,2,\dots$

\item[{\rm (DR)}] $\varphi_k(y)=\sqrt{\frac{4 \mu_k}{2 \mu_k - \sin(2\mu_k)}}\sin(\mu_k y)$,
$\lambda_k = \mu_k^2$, $k=1,2,\dots$, where \linebreak $\mu_k \in (k\pi-\pi/2,k\pi)$,
$k=1,2,\dots$, are roots $\widehat{d}(x):=q\sin(x)+x\cos(x)$.
Moreover, $\lim_{q\rightarrow 0} \mu_1(q)=\pi/2$ and $\lim_{q\rightarrow \infty} \mu_1(q)=\pi$.

\item[{\rm (DN)}] $\varphi_k(y)=\sqrt{2}\sin(\frac{2k+1}{2} \pi y)$,
$\lambda_k = \frac{(2k+1)^2}{4} \pi^2$, $k=0,1,2,\dots$

\item[{\rm (RR)}] $\varphi_k(y)=\sqrt{\frac{4 \tau_k}{(\tau_k^2-q^2)\sin(2\tau_k)+4q\tau_k \sin(\tau_k)^2+2\tau_k^3+2 q^2 \tau_k}}\bigl(q\sin(\tau_k y)+ \tau_k \cos(\tau_k y)\bigr)$,
$\lambda_k=\tau_k^2$, $k=1,2,\dots$, where
$\tau_k \in (0,\pi)$, $k=1,2,\dots$, are roots of $\widetilde{d}(x):=
2 q x \cos(x)+(q^2-x^2)\sin(x)$.
Moreover, $\lim_{q \rightarrow 0} \tau_1(q) = 0$
and $\lim_{q \rightarrow \infty} \tau_1(q) = \pi$.

\item[{\rm (NR)}] $\varphi_k(y)=\sqrt{\frac{4 \nu_k}{2 \nu_k + \sin(2\nu_k)}}\cos(\nu_k y)$,
$\lambda_k = \nu_k^2$, $k=1,2,\dots$, where \linebreak $\nu_k\in~((k-1)\pi,(k-\frac{1}{2})\pi)$,
$k=1,2,\dots$, are roots of $d(x):=x\sin(x)-q\cos(x)$. Moreover, $\lim_{q\rightarrow 0} \nu_1(q)=0$ and $\lim_{q\rightarrow \infty} \nu_1(q)=\pi/2$.

\item[{\rm (NN)}] $\varphi_k(y)=\sqrt{2}\cos(\pi k y)$, $\lambda_k=\pi^2 k^2$, $k=0,1,2,\dots$

\end{itemize}
\end{theorem}

\begin{proof}
If we multiply \eqref{Ciaramella_mini_10_eq:eig_1D} with $\varphi$, integrate over $[0,1]$,
and integrate by parts, we get $\lambda \int_0^1 |\varphi(y)|^2 dy
=\int_0^1 |\varphi'(y)|^2 dy-\varphi'(1)\varphi(1)+\varphi'(0)\varphi(0)$.
Using any of the above external conditions (and that $q>0$, for the Robin ones) 
one gets $\lambda \geq 0$. We refer to, e.g., \cite[Section 4.1]{Ciaramella_mini_10_Olver2013}
for similar discussions.
Now, all the cases can be proved by using the ansatz
$\varphi(y) = A \cos(\sqrt{\lambda}y)+B \sin(\sqrt{\lambda}y)$,
which clearly satisfies \eqref{Ciaramella_mini_10_eq:eig_1D}, and computing, e.g., $A$ and $\lambda$
in such a way that $\varphi(y)$ satisfies the two external conditions and $B$ as 
a normalization factor.
\end{proof}

The coefficients $\nu_1$, $\mu_1$ and $\tau_1$ as functions of $q$
are shown in Fig.~\ref{Ciaramella_mini_10_fig:2} (left),
\begin{figure}[t]
\centering
\includegraphics[scale=0.30]{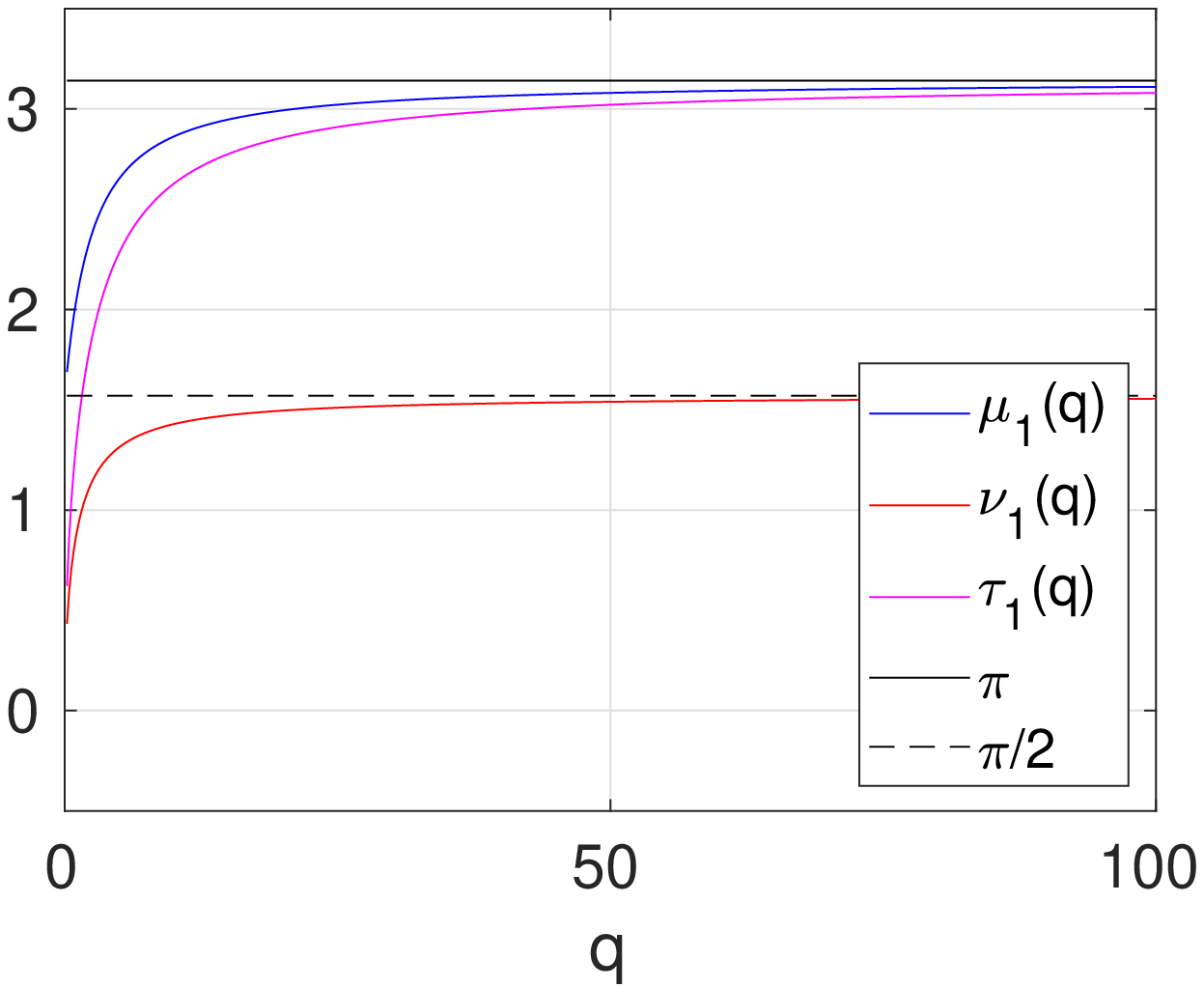}
\includegraphics[scale=0.30]{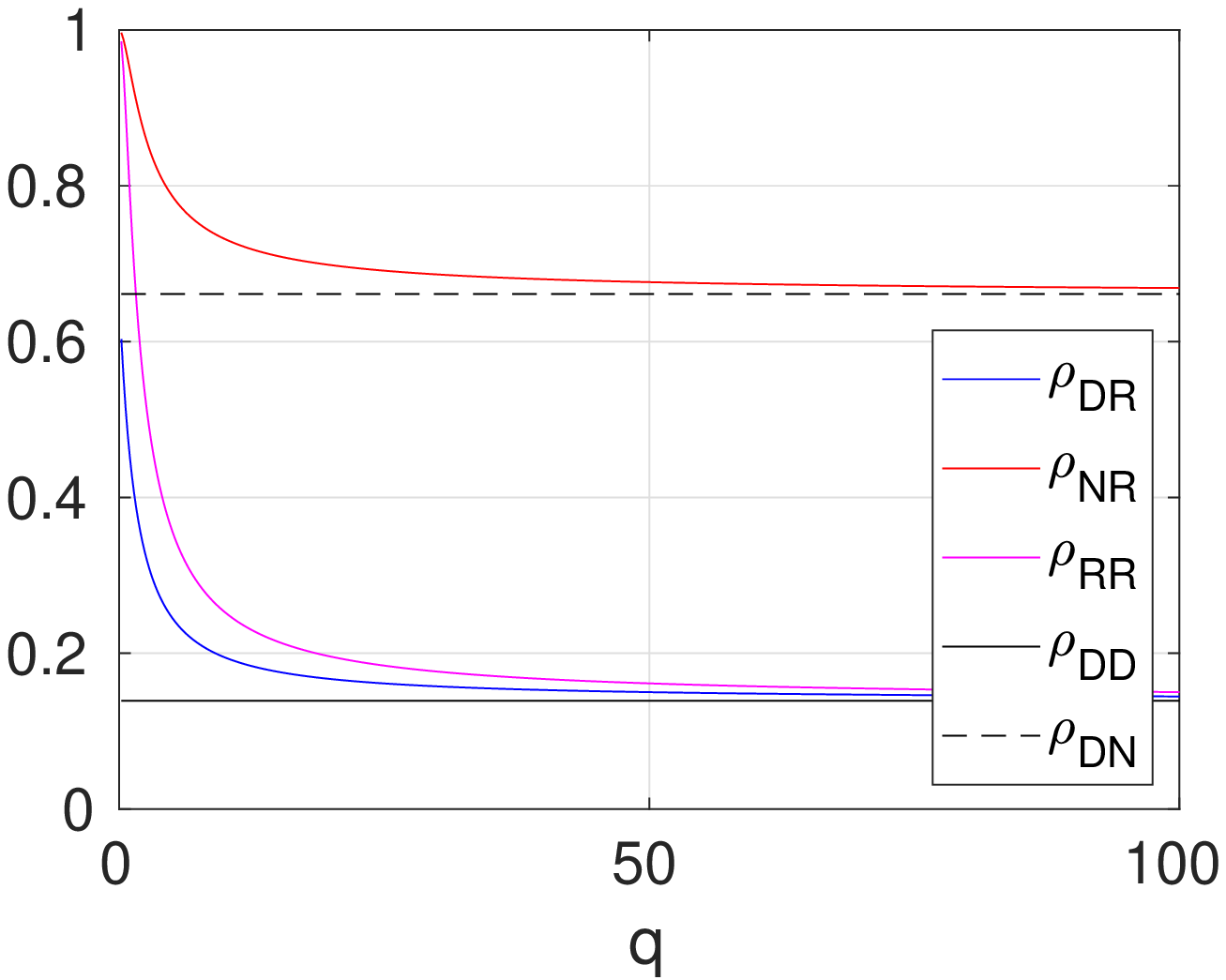}
\caption{Left: Maps $q\mapsto\mu_1(q)$, $q\mapsto\nu_1(q)$  and $q\mapsto\tau_1(q)$ .
Right: $\rho_{\rm DR}$, $\rho_{{\rm NR}}$, $\rho_{\rm DD}$ and $\rho_{\rm DN}$ as functions of $q$
and for $\delta=0.1$ and $L=1.0$.}
\label{Ciaramella_mini_10_fig:2}
\end{figure}
where we can observe that $\nu_1(q)< \frac{\pi}{2} < \mu_1(q) < \pi$ and $0 < \tau_1(q) < \pi$,
and that the maps $q\mapsto \nu_1(q)$, $q\mapsto \mu_1(q)$ and $q\mapsto \tau_1(q)$ 
increase monotonically and approach, respectively, $\frac{\pi}{2}$ and $\pi$ as $q \rightarrow \infty$.
Hence, by taking the limit $q \rightarrow 0$,
one can pass from the conditions {\rm (DR)}, {\rm (RR)} and {\rm (NR)} 
to {\rm (DN)}, {\rm (NN)} and {\rm (NN)}, respectively.
Similarly, by taking the limit $q \rightarrow \infty$,
the conditions {\rm (DR)}, {\rm (RR)} and {\rm (NR)} 
become {\rm (DD)}, {\rm (DD)} and {\rm (DN)}, respectively.

%%%%%%%%%%%%%%%%%%%%%%%%%%%%%%%%%%%%%%%%%%%%%%%%%%%%%%%%%%%%%%%%%%%%%%%%%%%%%%%%%%%%%%%%%%%%%%	
\section{Convergence and scalability}\label{Ciaramella_mini_10_sec:BC_scal}
\vspace*{-4mm}

Consider the Schwarz method \eqref{Ciaramella_mini_10_eq:chainP1_2D} and
any pair $(\mathcal{B}_b,\mathcal{B}_t)$ of operators as in Section \ref{Ciaramella_mini_10_sec:diag}.
The Fourier expansions of $u_j^n(x,y)$, $j=1,\dots,N$, are
\begin{equation}
u_j^n(x,y) = \sum_{k} \widehat{u}_j^n(x,\lambda_k) \varphi_k(y),
\end{equation} 
where the sum is over $k=1,2,\dots$ for {\rm (DD)}, {\rm (DR)},
{\rm (RR)} and {\rm (NR)}, and over $k=0,1,2,\dots$ for {\rm (DN)}
and {\rm (NN)}.
The functions $\varphi_k$ depend on the external boundary conditions
and are the ones obtained in Theorem~\ref{Ciaramella_mini_10_thm:1}.
The Fourier coefficients $\widehat{u}_j^n(x,\lambda_k)$ satisfy
\begin{equation}\label{Ciaramella_mini_10_eq:chainP1_1D}
\begin{split}
- \partial_{xx}\widehat{u}_j^n(x,\lambda_k)+\lambda_k \widehat{u}_j^n(x,\lambda_k) &= \widehat{f}_j(x,\lambda_k) \, \text{ in $(a_j,b_j)$}, \\
\mathcal{T}_{\ell}(\widehat{u}_j^n(\cdot,\lambda_k))(a_j)&= \mathcal{T}_{\ell}(\widehat{u}_{j-1}^{n-1}(\cdot,\lambda_k))(a_j),  \\
\mathcal{T}_r(\widehat{u}_j^n(\cdot,\lambda_k))(b_j) &= \mathcal{T}_r(\widehat{u}_{j+1}^{n-1}(\cdot,\lambda_k))(b_j) ,
\end{split}
\end{equation}
for $j=2,\dots,N$. For $j=1$, the condition at $a_1$ must be replaced by
$u_1^n(a_1) = 0$ and for $j=N$ the condition at $b_N$ must be replaced by
$u_N^n(b_N) = 0$.
If the operators $\mathcal{T}_{\ell}$ and $\mathcal{T}_r$ correspond
to Dirichlet conditions (see \eqref{Ciaramella_mini_10_eq:T_D}), then \eqref{Ciaramella_mini_10_eq:chainP1_1D} is a PSM.
If they correspond to Robin conditions (see \eqref{Ciaramella_mini_10_eq:T_R}), 
then \eqref{Ciaramella_mini_10_eq:chainP1_1D} is an OSM. The convergence of the iteration \eqref{Ciaramella_mini_10_eq:chainP1_1D}  
is analyzed in Theorem \ref{Ciaramella_mini_10_thm:2}.

\begin{theorem}[Convergence of Schwarz methods in Fourier space]\label{Ciaramella_mini_10_thm:2}
The contraction factors of the Schwarz methods\footnote{The contraction factor for 
\eqref{Ciaramella_mini_10_eq:chainP1_1D} (corresponding to the $k$-th Fourier component) 
is the spectral radius of the Schwarz iteration matrix; see \cite{Ciaramella_mini_10_CiaramellaGander,Ciaramella_mini_10_CiaramellaGander4}.} \eqref{Ciaramella_mini_10_eq:chainP1_1D}
are bounded by
\begin{equation}\label{Ciaramella_mini_10_eq:bound}
\rho(\lambda_k,\delta) =\frac{e^{2\lambda_k\delta}+e^{\lambda_k L}}{e^{2\lambda_k\delta+\lambda_k L}+1}.
\end{equation}
Moreover, it holds that $\rho(\lambda_k,\delta)\in[0,1]$
with $\rho(0,\delta)=1$ (independently of $N$), and that
$\lambda \mapsto \rho(\lambda,\delta)$ 
is strictly monotonically decreasing.
\end{theorem}

\begin{proof}
The Dirichlet case follows from 
\cite[Lemma 2 and Theorem 3]{Ciaramella_mini_10_CiaramellaGander}. See also
\cite[Lemma 2 and Theorem 1]{Ciaramella_mini_10_CiaramellaGander4}.
We focus here on the Robin case.
From Theorem 3 in \cite{Ciaramella_mini_10_CiaramellaGander4} and the corresponding
proof we have that the contraction factor of the OSM is bounded
by $\max\{ \varphi(\lambda,\delta,p) , |\zeta(\lambda,\delta,p)| \}$
where
\begin{equation*}
\begin{split}
\varphi(\lambda,\delta,p) &:= 
\frac{(\lambda+p)^2 e^{2\delta \lambda} -(\lambda-p)^2e^{-2\delta \lambda}+(\lambda+p)|\lambda-p|(e^{\lambda L}-e^{-\lambda L})}{(\lambda+p)^2 e^{\lambda L+2\lambda\delta} -(\lambda-p)^2 e^{-\lambda L-2\lambda\delta}} \geq 0,  \\
 \zeta(\lambda,\delta,p) &:= \frac{(\lambda+p)e^{-\lambda L}+(\lambda-p)e^{\lambda L}}{(\lambda+p)e^{\lambda(L+2\delta)}+(\lambda-p)e^{-\lambda(L+2\delta)}},
\end{split}
\end{equation*}
with $\varphi(\lambda,\delta,p) \leq \varphi(\lambda,\delta,0)=
\lim_{\widetilde{p} \rightarrow \infty}\varphi(\lambda,\delta,\widetilde{p})= \frac{e^{2\delta \lambda}-e^{-2\delta \lambda} +e^{\lambda L} -e^{-\lambda L}}{e^{\lambda L+2\delta \lambda}-e^{-\lambda L-2\delta \lambda}}$ for all $\lambda \geq 0$ and $\delta>0$.
If we compute the derivative of $\lambda\mapsto\varphi(\lambda,\delta,0)$
we get
\begin{equation*}
\partial_{\lambda} \varphi(\lambda,\delta,0)
= - \frac{L(e^{4 \delta \lambda + L \lambda}-e^{L \lambda}) + 2 \delta (e^{2 \delta \lambda +  2 L \lambda}-e^{2 \delta \lambda}) }{(e^{2 \delta \lambda+L \lambda}+1)^2},
\end{equation*}
which is negative for any $\lambda \geq 0$ and $\delta>0$.
Thus, $\lambda \mapsto \varphi(\lambda,\delta,0)$ is strictly monotonically
decreasing.
Let us now study the function $\zeta(\lambda,\delta,p)$.
Direct calculations reveal that 
$\partial_p \zeta(\lambda,\delta,p) = -\frac{2\lambda e^{2\delta \lambda}(e^{4\lambda(\delta + L)}-1)}{((\lambda+p) e^{4 \delta \lambda +2 L \lambda} + \lambda-p)^2}$,
which is negative for any $\lambda \geq 0$ and $\delta>0$,
and $\zeta(\lambda,\delta,0)=\frac{(e^{2 L \lambda}+1)e^{2 \delta \lambda}}{e^{4 \delta \lambda +2 L \lambda}+1}>0$ and $\lim_{p \rightarrow \infty} \zeta(\lambda,\delta,p) = - \frac{(e^{2 L \lambda}-1)e^{2 \delta \lambda}}{e^{4 \delta \lambda +2 L \lambda}-1}<0$
for any $\lambda \geq 0$ and $\delta>0$. These observations imply that
$p\mapsto \zeta(\lambda,\delta,p)$ is strictly monotonically decreasing
and attains its maximum at $p=0$.
Finally, a direct comparison shows that
$\varphi(\lambda,\delta,0) \geq \zeta(\lambda,\delta,0)\geq\lim_{p \rightarrow \infty}\left|\zeta(\lambda,\delta,p)\right|$
and the result follows, because
$\varphi(\lambda,\delta,0)=\frac{e^{2\delta \lambda}-e^{-2\delta \lambda} +e^{\lambda L} -e^{-\lambda L}}{e^{\lambda L+2\delta \lambda}-e^{-\lambda L-2\delta \lambda}}=\frac{e^{2\lambda\delta}+e^{\lambda L}}{e^{2\lambda\delta+\lambda L}+1}$.
\end{proof}

Theorem \ref{Ciaramella_mini_10_thm:2} gives the same bound \eqref{Ciaramella_mini_10_eq:bound} for 
the convergence factors of PSM and OSM.
This fact is not surprising. First, it is well known that OSM converges faster than PSM
for $\delta>0$. Hence, a convergence bound for the PSM is a valid bound also for the OSM.
Second, in the above proof the convergence bound for the OSM
is obtained for $p \rightarrow \infty$, which corresponds to passing from Robin transmission
conditions to Dirichlet transmission conditions.
The bound \eqref{Ciaramella_mini_10_eq:bound} is based on the ones
obtained in \cite{Ciaramella_mini_10_CiaramellaGander,Ciaramella_mini_10_CiaramellaGander4}. 
These are quite sharp for large values of $N$; see, e.g., 
\cite[Fig. 4 and Fig. 5]{Ciaramella_mini_10_CiaramellaGander4}.

We can now prove our main convergence result, which allows us
to study convergence and scalability of PSM and OSM 
for all the external conditions considered in Section \ref{Ciaramella_mini_10_sec:diag}.

\begin{theorem}[Convergence of PSM and OSM]\label{Ciaramella_mini_10_thm:3}
The contraction factors (in the $L^2$ norm) of PSM and OSM for the 
solution to \eqref{Ciaramella_mini_10_eq:P_2D} are bounded by
\begin{align*}
&{\rm (DD)} \; \rho_{\rm DD}(\delta) := \rho(\pi^2,\delta),
&&{\rm (DR)} \; \rho_{\rm DR}(\delta,q) := \rho(\mu_1(q)^2,\delta), \\
&{\rm (DN)} \; \rho_{\rm DN}(\delta) := \rho(\pi^2/4,\delta),
&&{\rm (RR)} \; \rho_{\rm RR}(\delta,q) := \rho(\tau_1(q)^2,\delta), \\
&{\rm (NR)} \; \rho_{{\rm NR}}(\delta,q) := \rho(\nu_1(q)^2,\delta),
&&{\rm (NN)} \; \rho_{\rm NN}(\delta) := \rho(0,\delta)=1,
\end{align*}
where $q \in (0,\infty)$ and $\rho(\lambda,\delta)$ is defined in Theorem \ref{Ciaramella_mini_10_thm:2}.
Moreover, for any $\delta>0$ we have that
\begin{equation}\label{Ciaramella_mini_10_eq:chain}
\begin{split}
&\rho_{\rm DD}(\delta)
<\rho_{\rm DR}(\delta,q)<\rho_{\rm DN}(\delta)
<\rho_{{\rm NR}}(\delta,q)<\rho_{\rm NN}(\delta)=1,
\end{split}
\end{equation}
\vspace*{-15mm}

\begin{equation}\label{Ciaramella_mini_10_eq:chain2}
\begin{split}
&\rho_{\rm DD}(\delta) < \rho_{\rm RR}(\delta,q) < \rho_{\rm NN}(\delta)=1.
\end{split}
\end{equation}
\end{theorem}

\begin{proof}
According to Theorem \ref{Ciaramella_mini_10_thm:2}, the bounds of the Fourier contraction factor $\rho(\lambda,\delta)$ is monotonically decreasing in $\lambda$.
Therefore, an upper bound for the convergence factor of PSM and OSM (in the $L^2$ norm) can be obtained by taking the maximum over the admissible Fourier frequencies $\lambda_k$
and invoking Parseval's identity (see, e.g., \cite{Ciaramella_mini_10_CiaramellaGander}).
Recalling Theorem \ref{Ciaramella_mini_10_thm:1}, these maxima are attained at $\lambda_1=\pi^2$ for {\rm (DD)},
$\lambda_1=\mu_1^2$ for {\rm (DR)}, 
$\lambda_0=\pi^2/4$ for {\rm (DN)}, 
$\lambda_1=\tau_1^2$ for {\rm (RR)}, 
$\lambda_1=\nu_1^2$ for {\rm (NR)}, and
$\lambda_0=0$ for {\rm (NN)}.
The inequalities \eqref{Ciaramella_mini_10_eq:chain} and \eqref{Ciaramella_mini_10_eq:chain2} follow from the
monotonicity $\lambda \mapsto \rho(\lambda,\delta)$
and the fact that $\nu_1(q)<\frac{\pi}{2}<\mu_1(q)<\pi$
and $\tau_1(q) \in (0,\pi)$.
\end{proof}

The inequalities \eqref{Ciaramella_mini_10_eq:chain} and \eqref{Ciaramella_mini_10_eq:chain2} imply that
the contraction factor is bounded, independently of $N$, by a constant strictly smaller than $1$
for all the cases except {\rm (NN)}.
In the case {\rm (NN)}, the first Fourier frequency is $\lambda_0=0$.
Hence, the coefficients $\widehat{u}_j^n(x,\lambda_0)$ are generated
by the 1D Schwarz method
\begin{equation}\label{Ciaramella_mini_10_eq:chainP1_1D2}
\begin{split}
- \partial_{xx}\widehat{u}_j^n(x,\lambda_0) &= \widehat{f}_j(x,\lambda_0) \, \text{ in $(a_j,b_j)$}, \\
\mathcal{T}_{\ell}(\widehat{u}_j^n(\cdot,\lambda_0))(a_j)&= \mathcal{T}_{\ell}(\widehat{u}_{j-1}^{n-1}(\cdot,\lambda_0))(a_j),  \\
\mathcal{T}_r(\widehat{u}_j^n(\cdot,\lambda_0))(b_j) &= \mathcal{T}_r(\widehat{u}_{j+1}^{n-1}(\cdot,\lambda_0))(b_j) ,
\end{split}
\end{equation}
which is known to be not scalable; see, e.g., \cite{Ciaramella_mini_10_CiaramellaGander4,Ciaramella_mini_10_CHS1}.
The scalability of PSM and OSM for different external conditions applied at the top and
at the bottom of the domain is summarized in Table \ref{Ciaramella_mini_10_tab:1}.
\begin{table}[t]
%\centering
\setlength{\tabcolsep}{3.5pt}
\begin{tabular}{ l|c|c|c } 
 \diag{.0em}{1.0cm}{\footnotesize{bottom}}{\footnotesize{top}}& Dirichlet & Robin & Neumann \\ \hline
 Dirichlet & yes & yes & yes \\ \hline
 Robin & yes & yes & yes \\ \hline
 Neumann & yes & yes & no \\  \hline
\end{tabular}
\setlength{\tabcolsep}{3.5pt}
\begin{tabular}{ l|c|c|c } 
 \diag{.0em}{1.0cm}{\footnotesize{bottom}}{\footnotesize{top}}& Dirichlet & Robin & Neumann \\ \hline
 Dirichlet & - & yes & - \\ \hline
 Robin & yes & no & no \\ \hline
 Neumann & - & no & - \\  \hline
\end{tabular}
\caption{Left: Scalability of PSM and OSM for different external conditions (for a fixed and finite $q>0$) 
applied at the top and at the bottom of the domain.
Right: Robustness of PSM and OSM with respect to $q \in [0,\infty]$.}
\label{Ciaramella_mini_10_tab:1}
\end{table}
Inequalities \eqref{Ciaramella_mini_10_eq:chain} and \eqref{Ciaramella_mini_10_eq:chain2} lead to another interesting observation.
The contraction factors are clearly influenced by the external boundary conditions.
Dirichlet conditions lead to faster convergence than Robin conditions, which in turn
lead to faster convergence than Neumann conditions. For example, if one external condition
is of the Dirichlet type, then PSM and OSM converge faster if the other condition
is of the Dirichlet type and slower if this is of Robin and even slower for the Neumann type.
The case {\rm (RR)} is slightly different, because the corresponding convergence of PSM
and OSM depends heavily on the Robin parameter $q$.
The behavior of the bounds $\rho_{\rm RR}(\delta,q)$, $\rho_{\rm DR}(\delta,q)$ 
and $\rho_{\rm NR}(\delta,q)$ with respect to $q$ is depicted in Fig. \ref{Ciaramella_mini_10_fig:2} (right),
which shows the bounds discussed in Theorem \ref{Ciaramella_mini_10_thm:3} as functions
of $q$ (recall that $\rho_{\rm NN}=1$).
Here, we can observe that the inequalities \eqref{Ciaramella_mini_10_eq:chain} and \eqref{Ciaramella_mini_10_eq:chain2} 
are satisfied and that
\begin{itemize}\itemsep0em
\item As $q$ increases
the Dirichlet part of the Robin external condition dominates. In addition, the bounds
$\rho_{\rm RR}$ and $\rho_{\rm DR}$ decrease and approach
$\rho_{\rm DD}$ as $q \rightarrow \infty$. 
Similarly, $\rho_{\rm NR}$ decreases and approaches $\rho_{\rm DN}$.
\item As $q$ decreases the Neumann part of the Robin external condition 
dominates. In addition, the bounds
$\rho_{\rm NR}$ and $\rho_{\rm RR}$ decrease and approach
$\rho_{\rm NN}=1$ as $q \rightarrow 0$.
Similarly, $\rho_{\rm DR}$ increases and approaches $\rho_{\rm DN}$.
\end{itemize}
These observations lead to Tab. \ref{Ciaramella_mini_10_tab:1} (right), where we summarize the robustness
of PSM and OSM with respect to the parameter $q$. The methods are robust 
with respect to $q$ only if one of the two external boundary conditions is
of Dirichlet type. This is due to the fact that Robin conditions become Neumann conditions
for $q \rightarrow 0$.

%%%%%%%%%%%%%%%%%%%%%%%%%%%%%%%%%%%%%%%%%%%%%%%%%%%%%%%%%%%%%%%%%%%%%%%%%%%%%%%%%%%%%%%%%%%%%%	
\section{Numerical experiments}\label{Ciaramella_mini_10_sec:numerics}
\vspace*{-4mm}

In this section, we test the scalability of PSM and OSM by numerical simulations.
For this purpose, we run PSM and OSM for all the external boundary conditions discussed in
this paper and measure the number of iterations required to reach a tolerance on
the error of $10^{-6}$. To guarantee that the initial errors contain all frequencies,
the methods are initialized with random initial guesses.
In all cases, each subdomain is discretized with a uniform grid of
size $90$ interior points in direction $x$ and $50$ interior points in direction $y$.
The mesh size is $h = \frac{L}{51}$, with $L=1$, and the overlap parameter is $\delta=10h$. 
For the OSM the robin parameter is $p=10$. The Robin parameter $q$ of the external Robin conditions
is $q=10$, and the {\rm (RR)} case is also tested with $q=0.1$. 
The results of our experiments are shown in Tab. \ref{Ciaramella_mini_10_tab:2}
and confirm the theoretical results discussed in the previous sections.
\begin{table}[t]
\centering
\setlength{\tabcolsep}{3.5pt}
\begin{tabular}{ c|c|c|c|c|c|c|c } 
$N$ & {\rm DD} & {\rm DR}(10) & {\rm DN} & {\rm RR}(10) & {\rm NR}(10) & {\rm NN} & {\rm RR}(0.1) \\ \hline
3  &  12 - 9 &   13 - 10 &   27 - 19 &   14 - 10 &   26 - 19 &   77  - 54 &   65 - 45\\
4  &  13 - 9 &   14 - 10 &   29 - 21 &   15 - 11 &   29 - 21 &  130  - 90 &  95 - 66\\
5  &  13 - 9 &   14 - 10 &   31 - 22 &   15 - 11 &   31 - 22 &  194  - 134 &  124 - 86 \\
10 &  13 - 10 &  14 - 10 &   33 - 24 &   15 - 11 &   34 - 24 &  $>$401 - $>$401 &  227 - 155\\
20 &  13 - 10 &  14 - 10 &   34 - 24 &   15 - 11 &   35 - 24 &  $>$401 - $>$401 &  293 - 199\\
30 &  13 - 10 &  14 - 10 &   34 - 24 &   15 - 11 &   35 - 24 &  $>$401 - $>$401 &  311 - 210\\
40 &  13 - 10 &  14 - 10 &   34 - 24 &   15 - 11 &   35 - 24 &  $>$401 - $>$401 &  317 - 214\\
50 &  13 - 10 &  14 - 10 &   34 - 24 &   15 - 11 &   35 - 24 &  $>$401 - $>$401 &  319 - 216\\ \hline
\end{tabular}
\caption{Number of iterations of PSM (left) and OSM (right) needed to reduce the norm of the error
below a tolerance of $10^{-6}$ for increasing number $N$ of fixed-sized subdomains.
The maximum number of allowed iterations is 401. This limit is only reached in the {\rm (NN)} case, 
for which PSM and OSM are not scalable.}
\label{Ciaramella_mini_10_tab:2}
\end{table} 
	
	\vspace*{-3mm}
	\bibliographystyle{plain}
	\bibliography{Ciaramella_mini_10}
\end{document}